%% file: 97-Complexity-v4.tex
\documentclass[11pt]{amsart}
\usepackage{geometry}                
\usepackage[parfill]{parskip}    
\usepackage{graphicx}
\usepackage{amssymb}
\usepackage{epstopdf}
\usepackage[applemac]{inputenc}
\usepackage[a4paper=true,%
breaklinks=true,%
colorlinks=true,%
linkcolor={blue},%
citecolor={red},%
pdftitle={Computing the Complexity of  Natural Numbers},%
pdfauthor={J. Arias de Reyna},%
pdfkeywords={snapshot},%
pdfstartview=FitH,
]{hyperref}
\include{pythonlisting}

\theoremstyle{plain}
\newtheorem{theorem}{Theorem}
\newtheorem{proposition}[theorem]{Proposition}

\newtheorem{corollary}[theorem]{Corollary}

\theoremstyle{definition}
\newtheorem{definition}[theorem]{Definition}

\theoremstyle{remark}
\newtheorem{remark}[theorem]{Remark}

\def\N{\ensuremath {{\bf N}}}

\newcommand{\kMax}{\mathop{\rm kMax }}
\newcommand{\nMax}{\mathop{\rm nMax }}
\newcommand{\cMax}{\mathop{\rm cMax }}
\newcommand{\Compl}{\mathop{\rm Compl }}
\newcommand{\newH}{\mathop{\rm newH }}
\newcommand{\step}{\mathop{\rm step }}
\newcommand{\Heads}{\mathop{\rm Heads }}
\newcommand{\Shift}{\mathop{\rm Shift }}
\newcommand{\CalculateRunningBlock}{\mathop{\rm CalculateRunningBlock }}
\newcommand{\shift}{\mathop{\rm shift }}
\newcommand{\shiftA}{\mathop{\rm shiftA }}
\newcommand{\shiftB}{\mathop{\rm shiftB }}
\newcommand{\Block}{\mathop{\rm Block }}
\newcommand{\BlockA}{\mathop{\rm BlockA }}
\newcommand{\BlockB}{\mathop{\rm BlockB }}
\newcommand{\kMaxfor}{\mathop{\rm kMaxfor }}
\newcommand{\Products}{\mathop{\rm Products }}

\newcommand{\complexity}{\mathop{\rm complexity}}

\newfont{\cmbsy}{cmbsy10}
\newcommand{\Orden}{\mathop{\hbox{\cmbsy O}}\nolimits}

\title{Algorithms for determining integer complexity}

\author{J. Arias de Reyna}
\address{Facultad de Matem\'aticas, 
Universidad de Sevilla, \newline
Apdo.~1160, 41080-Sevilla, Spain}
\email{arias@us.es}

\author{J. van de Lune}
\address{Langebuorren 49, 9074 CH Hallum, The Neterlands\newline (Formerly at
CWI, Amsterdam ) }
\email{j.vandelune@hccnet.nl}

\begin{document}
\maketitle
\section*{Introduction.}

The complexity $\Vert n\Vert$ of a natural number $n$ is defined as the least 
number of $1$'s needed to express the positive integer $n$ using only the 
two operations $+$ and $\times$
and parentheses. At present there is a growing literature about this topic. 
For details we refer to \cite{MP}, \cite{G}, \cite{R}, \cite{A}, \cite{AZ1}, \cite{JI}. 

We present three algorithms to compute the complexity $\Vert n\Vert$ of all natural 
numbers  $ n\le N$.  The first of them is a brute force algorithm, computing 
all these complexities in time $\Orden(N^2)$ and space $\Orden(N\log^2 N)$. 
The main problem of this algorithm is the time needed for the computation. 
In 2008 there appeared three independent solutions to this  problem:
\cite{SS}, \cite{MNF}, \cite{AL}. All three are very similar. Only \cite{SS}
gives an estimation of the performance of its algorithm, proving that the
algorithm computes the complexities in time $\Orden(N^{1+\beta})$, where 
$1+\beta =\log3/\log2\approx1.584963$. The other two algorithms, presented in 
\cite{MNF} and \cite{AL}, were very similar but both superior to the one in \cite{SS}.
In Section \ref{S:2} we  present a version of these algorithms  and in Section 
\ref{S:UB} it is shown that they run in time $\Orden(N^\alpha)$ and space 
$\Orden(N\log\log N)$. (Here $\alpha = 1.230175$). 

The authors of the present paper, at the moment of sending \cite{AL} to a journal 
(Oct. 2008)
saw Martin N. Fuller's program \cite{MNF} posted in the OEIS.
Recognizing 
the similarity of our own approach in \cite{AL} we contacted  Fuller and proposed
to publish  our results jointly. This started a very active collaboration,
improving the results in \cite{MNF} and \cite{AL}. 
The main problem with the algorithm presented in Section \ref{S:2} concerned the
space requirement. In \cite{MNF} Fuller
has given the idea of how to reduce the space requirements. 
He then programmed this improved algorithm and used it to compute 
the complexities for all $n\le 10^{11}$.  

For some reason (unknown to us) there was an interruption of this collaboration 
after which we 
could not contact Fuller any more.
The present paper is the result of the above collaboration and we both think that 
Fuller should have been the main author of this paper.  Of course he is not 
responsible for any errors contained in it.

In Section \ref{S:2} we present the algorithm of \cite{MNF} and \cite{AL}. The main 
advantage of this algorithm with respect to that in \cite{SS} is the definition 
of $\kMax$ in Section \ref{2:7}. This explains  the difference in performance from 
$\Orden(N^{1+\beta})$ to $\Orden(N^\alpha)$.

In Section \ref{S:F}
we present a detailed description of this space-improved  algorithm of Fuller 
and in Section \ref{S:PF} we prove that it runs in time $\Orden(N^\alpha)$
and space $\Orden(N^{(1+\beta)/2}\log\log N)$, where 
$\alpha=1.230175$ and $(1+\beta)/2\approx0.792481$. 

We have delayed the publication of the present paper in the hope of again getting
in contact with M.  Fuller. 
In the mean time it has been published \cite{JI} where the 
computation of the complexities of all $n\le 10^{12}$ is announced.  
In \cite{JI}
it is said that they have used an algorithm inspired by the description in \cite{MNF} of 
the space-improved algorithm of Fuller, but do not give a detailed description of 
their algorithm.  The  results proved in this paper about the 
performance of the algorithms
are not trivial and are still  missing.

There are many known  results about the complexity of natural numbers. We will use the inequality $\Vert n\Vert\le 3\log_2n$ for $n>1$, which appeared first in \cite{G}.

Since  we try to explain two algorithms, we have tried to concentrate on explaining
things to human beings how to proceed instead of trying to explain to a computer. This is the definition of \emph{Literate programming}. We follow the recommendations
of Knuth in \cite{K}. 

The full power of these programs will be apparent  when implemented in a 
language such as C or Pascal, 
but to explain the algorithm we have preferred Python as a language. In this way 
we can concentrate on the computations needed and not on the manipulation of  data. 

To one not used to Python it is a very simple language, 
the main difficulty possibly being that
the scopes of the for, while, etc. loops are indicated by  indentations. We also 
note 
that for integers $a$ and $b$  Python defines $a/b$ as $\lfloor \frac{a}{b}\rfloor$.

\section{Brute force algorithm.}\label{Stwo}

\xdef\frame{frame=none, xleftmargin = 0.5cm, xrightmargin=0.5cm}

This is the algorithm  used in \cite{A} to compute the values 
of $\Vert n\Vert$ for $1\le n\le \nMax$ with $\nMax = 200\,000$.

It is based on the fact that $\Vert 1\Vert =1$, and (also see Section \ref{S2:1n})
\begin{equation}
\Vert n\Vert = \min_{\substack{a,b<n\in\N\\ a+b=n\text{ or } ab=n}}
\Vert a\Vert+\Vert b\Vert.
\end{equation}
\bigskip

\lstset{escapeinside={(*@}{@*)}}

\begin{python}[frame=none, xleftmargin = 0.5cm,
xrightmargin=0.5cm ]
<(*@\ref{Stwo}@*) Brute force algorithm>  (*@$\equiv$@*)
nMax = 20000
Compl={}
s = range(nMax+1)
divisors = [[1] for n in s]
for k in range(2, nMax+1):
    j=1
    m = k
    while m <= nMax:
        divisors[m].append(k)
        m = m+k
Compl[1]=1
for n in range(2, nMax+1):
    S=nMax+1
    for k in range(1,1+n/2):
        a = Compl[k]+Compl[n-k]
        if a < S:
            S=a
    div_n = divisors[n]
    tau_n = len(div_n)
    P = nMax+1
    for k in range(1,tau_n-1):
        d = div_n[k]
        a = Compl[d]+Compl[n/d]
        if a < P:
            P = a
    Compl[n] = min(S,P)
\end{python}

We will not explain this simple minded algorithm. It is not very useful. At 
the end $\Compl[\,n\,]$ will give us the value of $\Vert n\Vert$ for 
$1\le n\le \nMax$. Observe that for each $n$ we compute
\begin{displaymath}
S := \min_{\substack{a,b<n\in\N\\ a+b=n}}
\Vert a\Vert+\Vert b\Vert\quad\text{and}\quad P:=
\min_{\substack{a,b<n\in\N\\ ab=n}}
\Vert a\Vert+\Vert b\Vert.
\end{displaymath}
The complexity of $n$ is then the  minimum of these two quantities:
$\Vert n\Vert=\min(S, P)$. 

It is easy to prove the following
\begin{proposition}
The above brute force algorithm computes the values of $\Vert n\Vert$ for $1\le n\le N$ 
in time $\Orden(N^2)$ and space $\Orden(N\log^2 N)$.
\end{proposition}

To store the complexities we need only $\Orden(N\log\log N)$ space. But our program
stores the divisors of all $1\le n\le N$ requiring $\Orden(N\log^2N)$ space. 

\section{Time Improved Algorithm.}\label{S:2}

\subsection{Idea of the algorithm}\label{S2:1n}
The main disadvantage of the brute force algorithm is its running time
$\Orden(N^2)$. This time is spent mainly on checking all the instances 
of $\Vert n-k\Vert+\Vert k\Vert$ for $1\le k\le n/2$. We will improve the time
by finding a suitable $\kMax$ such that 
\begin{equation}\label{mainPropertykMax}
S=\min_{1\le k\le n/2}\Vert k\Vert+\Vert n-k\Vert=\min_{1\le k\le\, \kMax}\Vert k\Vert+\Vert n-k\Vert.
\end{equation}

In this new program we will define a function: `$\complexity(\cdot)$'. Given an integer
$\nMax$, the output of $B = \complexity(\nMax)$ will be a  list $B$ 
containing the values of
the complexities of the numbers $1\le n\le \nMax$. More precisely we will have
$B[\,n\,]=\Vert n\Vert$.

\subsection{Plan of the second algorithm.}\label{2.2}
After some initializations the program will consist of only one main loop.
In this loop  we will compute successively the complexities of 
the natural numbers $2\le n\le \nMax$. The first part of this computation 
will consist of the computation of an adequate $\kMax$ for the corresponding $n$.
The computation of $\kMax$ uses an auxiliary function $E$. 

We import the module \emph{math} because we are using the mathematical function 
log in our program. 

\text{$\langle$\ref{2.2} Time improved algorithm $\rangle$}
\begin{python}[frame=none, xleftmargin = 0.5cm,
xrightmargin=0.5cm ]
<(*@\ref{2:3}@*) definition of E>
from math import *
def complexity(nMax):
    <(*@\ref{2:4}@*) Initializing Compl>
    <(*@\ref{2:5}@*) Main Loop>
    return Compl
\end{python}

\subsection{Definition of  $E$.}\label{2:3}
For each natural number $k$ let $E(k)$ be the largest number with complexity $k$ 
(see \cite{AZ1}).  Then $E(1)=1$, and for $j\ge1$,   
$E(3j)=3^j$, $E(3j+1)=4\cdot3^{j-1}$, 
$E(3j-1)=2\cdot 3^{j-1}$. We  extend this definition by $E(0):=1$. In this 
way $E$ coincides with the sequence A000792 in the OEIS \cite{OEIS}. We have
\begin{equation}
3^{(n-1)/3}\le E(n)\le  3^{n/3}\qquad n\ge0.
\end{equation}
Note that $E$ is a non decreasing function.
The following algorithm computes $E$:
\begin{python}[frame=none, xleftmargin = 0.5cm,
xrightmargin=0.5cm ]
<(*@\ref{2:3}@*) definition of E>
def E(n):
    if n == 0:
        return 1
    result = 1
    while n > 4:
        result *= 3
        n -= 3
    return result * n
\end{python}
For later application we  state the following easily proved fact.
\begin{proposition}
The cost of computing $E(n)$ is $\Orden(n)$ operations and $\Orden(n)$ space.
\end{proposition}
\begin{proof}
Observe that to write $3^{n/3}$ we need $\Orden(n)$ space. 
\end{proof}

\subsection{Initializing  $\Compl$.}\label{2:4}

The complexity function outputs a list $\Compl\,[\,\cdot\,]$. We first fill this list with 
a value which is an upper bound for all the true complexities. 
This  upper bound of $\Compl\,[\,n\,]$ will be denoted by $\cMax$.
Since $\Vert n\Vert\le \frac{3}{\log2}\log n$, we may take
this bound equal to  $\lfloor 3\log \nMax/\log2\rfloor$. We take a unit more
because in  case $\nMax$ is a power of $2$ we are near the singularity of 
the $\lfloor\cdot\rfloor$ function.

We also initialize  $\Compl[\,1\,]=1$.

\begin{python}[frame=none, xleftmargin = 0.5cm,
xrightmargin=0.5cm ]
<(*@\ref{2:4}@*) Initializing Compl>  (*@$\equiv$@*)
cMax = int(3*log(nMax)/log(2))+1
Compl =[cMax for n in range(0,nMax+1)]
Compl[1] = 1
\end{python}

\subsection{Main Loop.}\label{2:5}

Now we have the correct value of $\Compl\,[\,1\,]$. Dynamically we compute $\Compl\,[\,n\,]$
for $2\le n\le \nMax$.
The pertinent loop consists of the following parts:

\begin{python}[frame=none, xleftmargin = 0.5cm,
xrightmargin=0.5cm ]
<(*@\ref{2:5}@*) Main Loop>  (*@$\equiv$@*)
for n in range(2,nMax+1):
    <(*@\ref{2:6}@*) usual best value>
    <(*@\ref{2:7}@*) computing kMax>
    <(*@\ref{2:8}@*) testing the sums>
    <(*@\ref{2:9}@*) testing the products>
\end{python}

\subsection{Usual best value.}\label{2:6}
When computing $\Vert n\Vert$ we already have $\Compl\,[\,j\,]=\Vert j\Vert$ 
for all $1\le j\le n-1$,
and for $k\ge n$ we have $\Compl\,[\,k\,]\ge\Vert k\Vert$. We will proceed
to check all sums  and products to replace $\Compl\,[\,k\,]$ by better values, so 
that eventually we will obtain the correct values.

We call  $\Vert n-1\Vert+1$ the \emph{usual best value} of $\Vert n\Vert$.
At this stage we simply replace $\Compl\,[\,n\,]$ by this bound if needed.

\begin{python}[frame=none, xleftmargin = 0.5cm,
xrightmargin=0.5cm ]
<(*@\ref{2:6}@*) Usual best value>  (*@$\equiv$@*)
a = Compl[n-1]+1
if a < Compl[n]:
    Compl[n] = a
\end{python}

\subsection{Computing  $\kMax$.}\label{2:7}
To compute $\kMax$ we use  a bound $s$  with
$\Vert n\Vert\le s$. Here we always  take  $s=\Vert n-1\Vert+1$.

Here is our  procedure to compute $\kMax$. Subsequently
we will prove that it has the desired properties. 

\begin{python}[frame=none, xleftmargin = 0.5cm,
xrightmargin=0.5cm ]
<(*@\ref{2:7}@*) Computing kMax>  (*@$\equiv$@*)
target =Compl[n-1]
t = target/2
while E(t)+E(target-t) < n:
    t = t-1
kMax = E(t)
\end{python}

\begin{proposition}\label{kMaxbound}
For $n\ge 3$ the above procedure stops in  finite time giving a value
\begin{equation}\label{kMaxIneq1}
1\le \kMax \le \frac{n}{2}\Bigl(1-\sqrt{1-\frac{4}{n^2}3^{\Vert n-1\Vert/3}}\;\Bigr).
\end{equation}
\end{proposition}

\begin{proof}
Observe that in the second line we obtain $t=\lfloor\Vert n-1\Vert/2\rfloor$
(by the convention of Python $m/2$ returns the integer part of the fraction $\frac{m}{2}$). 
Therefore, for $n\ge3 $  after the second line $t$ will be a natural number or $0$.
In each while-loop of the while  $t$ is replaced by $t-1$, but if we arrive at
$t=0$ we will reach the condition of the while loop
\begin{displaymath}
E(0)+E(\Vert n-1\Vert-0)\ge 1+(n-1)=n
\end{displaymath}
and the program will stop with $t\ge0$. Hence at  finite time.

For each natural number $t$ we will have 
\begin{displaymath}
E(t)+E(\Vert n-1\Vert-t)\le 3^{t/3}+3^{(\Vert n-1\Vert-t)/3}<n
\end{displaymath}
whenever $x^2-nx+3^{\Vert n-1\Vert/3} <0$ with $x=3^{t/3}$. 
For $n\ge29$, the roots $x_1$, $x_2$ of this polynomial in $x$ are real, because its discriminant
is $n^2-4\cdot 3^{\Vert n-1\Vert/3}>0$. In fact
\begin{displaymath}
2\cdot 3^{\Vert n-1\Vert/6}\le 2\exp\Bigl(\frac{\log 3}{6}\frac{3}{\log 2}
\log n\Bigr)<n, \qquad n\ge 29.
\end{displaymath}
We also have $0<x_1<x_2$ because $x_1 x_2>0$, so that $x_1$ and $x_2$ have the same 
sign and since $x_1+x_2=n$ they are positive.   

The initial value of $t=\lfloor\Vert n-1\Vert/2\rfloor$, so that
$t=\Vert n-1\Vert/2-\varepsilon$ with $\varepsilon=0$ or $1/2$.  For $n\ge 31$, 
we will have
\begin{displaymath}
3^{t/3}+3^{(\Vert n-1\Vert-t)/3}
= 3^{\Vert n-1\Vert/6-\varepsilon/3}+3^{\Vert n-1\Vert/6+\varepsilon/3}
\le(3^{-\varepsilon/3}+3^{\varepsilon/3})\exp\Bigl(\frac{\log3}{\log 4}\log n\Bigr)<n.
\end{displaymath}
It follows that this initial value of $t$ satisfies $x_1<3^{t/3}<x_2$ when $n\ge31$.
While $3^{t/3}>x_1$ we will get $E(t)+E(\Vert n-1\Vert -t)<n$, so that  we end
the while loop with a value  $t=t_e$ such that $3^{t_e}\le x_1$. 
Then
\begin{equation}
\kMax=E(t_e)\le 3^{t_e/3}\le x_1= \frac{n}{2}\Bigl(
1-\sqrt{1-4\cdot3^{\Vert n-1\Vert/3}/n^2}\Bigr).
\end{equation}
While computing the needed complexities one may check that \eqref{kMaxIneq1} 
is also true for $3\le n\le 31$.
\end{proof}

\begin{corollary}\label{kmaxbound}
For $n\ge 2$  we have $\kMax\le 2n^\beta$ where 
$\beta=\frac{\log3}{\log2}-1\approx0.584963$. 
\end{corollary}

\begin{proof}
We have $3^{\Vert n-1\Vert/3}< 3^{\log n/\log 2}=n^{1+\beta}$ and $4n^{\beta-1}<1$ for 
$n\ge 29$, so that in this case
\begin{displaymath}
\sqrt{1-\frac{4}{n^2}3^{\Vert n-1\Vert/3}}>\sqrt{1-4 n^{\beta-1}}>1-4 n^{\beta-1}.
\end{displaymath}
Therefore, by Proposition \ref{kMaxbound} we get 
\begin{displaymath}
\kMax\le \frac{n}{2}\Bigl(
1-\sqrt{1-\frac{4}{n^2}3^{\Vert n-1\Vert/3}}\Bigr)< 2n^\beta.
\end{displaymath}
For $2\le n\le 29$ we may check directly that $\kMax< 2n^\beta$  (for $2\le n\le 50$ we have
$\kMax=1$ except for $n=24$ and $n = 48$ for which $\kMax=2$).
\end{proof}
We have the following main result about $\kMax$. 
\begin{proposition}\label{P:kMax}
For $n\ge2$ we have
\begin{displaymath}
S:=\min_{1\le k\le n/2}\Vert k\Vert+\Vert n-k\Vert=\min_{1\le k\le \kMax}
\Vert k\Vert+\Vert n-k\Vert.
\end{displaymath}
\end{proposition}
\begin{proof}
If not there would be a $k_0$ with $1\le k_0\le n/2$ and
\begin{equation}\label{k0 exists}
\Vert k_0\Vert+\Vert n-k_0\Vert<\min_{1\le k\le \kMax}
\Vert k\Vert+\Vert n-k\Vert.
\end{equation}
In particular 
$\Vert k_0\Vert+\Vert n-k_0\Vert<
1+\Vert n-1\Vert$, so that 
$\Vert k_0\Vert+\Vert n-k_0\Vert\le \Vert n-1\Vert$.
One of the two numbers $k_0$ and $n-k_0$, let us call it $u$, would 
satisfy $\Vert u\Vert\le \Vert n-1\Vert/2$. Therefore, $\Vert u\Vert\le t_0$, 
the initial value of $t$ in our procedure to get $\kMax$, and 
\begin{displaymath}
\Vert u\Vert+\Vert n-u\Vert \le \Vert n-1\Vert.
\end{displaymath}
We will also have $\Vert u\Vert\le t_e$ the final value of $t$ so that, by 
definition $\kMax=E(t_e)$.  In the other case we would have  $t_e<\Vert u\Vert\le t_0$
so that $t':=\Vert u\Vert$ would be one of the values of $t$ in the procedure 
(not the last) and we 
would  therefore have
\begin{displaymath}
E(t')+E(\Vert n-1\Vert-t')<n.
\end{displaymath}
From this we get a contradiction
\begin{displaymath}
n = u+(n-u)\le E(\Vert u\Vert)+E(\Vert n-u\Vert)\le E(t')+E(\Vert n-1\Vert-t')<n. 
\end{displaymath}
Therefore, $u\le E(\Vert u\Vert)\le E(t_e)=\kMax$, so that by \eqref{k0 exists}
we get the contradiction
\begin{displaymath}
\Vert u\Vert+\Vert n-u\Vert=\Vert k_0\Vert+\Vert n-k_0\Vert<\min_{1\le k\le \kMax}
\Vert k\Vert+\Vert n-k\Vert\le \Vert u\Vert+\Vert n-u\Vert.
\end{displaymath}
\end{proof}

\subsection{Testing the sums.}\label{2:8}
By Proposition \ref{P:kMax} we only have to check 
$\Vert k\Vert+\Vert n-k\Vert<\Compl\,[\,n\,]$ for 
$1\le k\le \kMax$. In \emph{\ref{2:6} Usual best value}  
we have checked $k=1$.

We need not  check $k=2$, $k=3$, \dots, $k=5$.  The first value we will check is $k=6$. We explain:
In \cite{AZ1} it is shown
that the least value  $b$ such that 
\begin{displaymath}
S:=\inf_{1\le k\le n}\Vert k\Vert+\Vert n-k\Vert=\Vert b\Vert+\Vert n-b\Vert
\end{displaymath}
is a \emph{solid number}, where $b$ is a solid number if $b=u+v$ implies
$\Vert b\Vert<\Vert u\Vert+\Vert v\Vert$. 

It follows that we only need to test $\Vert k\Vert+\Vert n-k\Vert<\Compl\,[\,n\,]$ when
$k$ is a solid number. But the sequence of solid numbers starts with 
$1$, $6$, $8$, \dots 
Hence we arrive at the following algorithm to check the sums:

\begin{python}[frame=none, xleftmargin = 0.5cm,
xrightmargin=0.5cm ]
<(*@\ref{2:8}@*) Testing the sums>  (*@$\equiv$@*)
for m in range(6, kMax+1):
    sumvalue = Compl[m]+Compl[n-m]
    if sumvalue < Compl[n]:
        Compl[n] = sumvalue
\end{python}

\subsection{Testing the products.}\label{2:9}
Here we consider all the multiples $k\cdot n$ of $n$ and substitute the 
value of $\Compl\,[\,k\cdot n\,]$ if necessary.  Therefore, when we arrive at the 
case $n$
all its divisors will have been checked, so that we will have the correct
value of $\Compl\,[\,n\,]$. 

\begin{python}[frame=none, xleftmargin = 0.5cm,
xrightmargin=0.5cm ]
<(*@\ref{2:9}@*) Testing the products>  (*@$\equiv$@*)
for k in range(2, min(n, nMax/n)+1):
    prodvalue = Compl[k]+Compl[n]
    if prodvalue < Compl[k*n]:
        Compl[k*n] = prodvalue
\end{python}

\section{Fuller algorithm.}\label{S:F}
In practice the main limitation of the algorithm presented in Section \ref{Stwo} is 
the space requirement $\Orden(N\log\log N)$. We present here an idea of 
Fuller, partially expressed in \cite{MNF} to overcome this difficulty.

The output of the new program is different from that of Section \ref{Stwo}. Now
we will compute successively  the complexities of all numbers $n\le N$ but
we do not store all these values.  So the output of the program will 
consist of several statistics about the complexities of the numbers $1\le n\le N$. 
These statistics, which may change from run to run, 
will be computed 
simultaneously with the complexities. 

\subsection{The basic idea.} The algorithm will store the minimal data needed to calculate $\Vert n\Vert$
\begin{itemize}
\item{$\Vert a\Vert $ for $a\mid n$.}
\item{$\Vert k\Vert$ for $k\le \kMax$.}
\item{$\Vert n-k\Vert$ for $k\le \kMax$.}
\end{itemize}
Small divisors and small summands of $n$  will be stored in a \emph{fixed block}, 
containing $\Vert n\Vert$ up to some fixed upper limit $H$.  
Recent calculations will be stored in a \emph{running block} around $n$ 
which is used to retrieve $\Vert n-k\Vert$.  
Large divisors $\Vert n/2\Vert$, $\Vert n/3\Vert$, etc. will be calculated and 
stored in the same way as $n$, using \emph{additional running blocks}
and sharing the same fixed block. 

The program will have several parameters: $N$ (or $\nMax$ in the program) will 
be the limit to which we compute the complexities. 
We will have a fixed block where we will store the complexities for 
$1\le n\le H$. Thus $H$ is the length of the fixed block. 
We will have a number of running blocks $B_j$ all of them of length $L=2\ell$. 
$B_1$ will be the main running block, where we will store the last 
computed complexities. The other running blocks  $B_2$, $B_3$,  \dots\ will 
contain the complexities of the large divisors $n/2$, $n/3$, \dots\ of the 
numbers  immediately following $B_1$. 

All running blocks $B_j$ (for $1\le j\le N/H$), each of length $L$ will 
contain $\Vert n\Vert$ for the range $H_j-L<n\le H_j$,  where $H_j$ will 
start at $H$ and then move in steps of size $\ell$ to finish at or just above $N/j$.

For the computation to run smoothly we need that these parameters satisfy
\begin{equation}\label{E:ineq} 
\kMax < \ell,\quad N\le \ell H,\quad \ell\mid H,\quad \ell\mid N,\quad
H\ge\sqrt{LN}.
\end{equation}

The algorithm was used by Fuller in 2009 to compute $n$ up to 
$N=10^{11}$ using just over $10^9$ bytes of memory. 
The running time was 106 hours using 1 processor of a 2.8 GHz dual core PC. 
The parameters used for this run 
were $N=10^{11}$, $H=10^9$, $L=10^6$. 
Sieving was used to find potential factors $2\le a\le b$, $H_j-\ell<ab\le H_j$. 
Thus there was a performance gain by using a large $L$. The program runs roughly 
$1/3$ of the speed when using $L=10^5$ instead of $L=10^6$.

\subsection{General description}\label{3:2}
Roughly  the program executed the following steps:

\begin{itemize}
\item[1.]{Calculate $\Vert n\Vert$ for the fixed block $B_0$,  using  the earlier algorithm of Section \ref{S:2}.}
\item[2.]{Copy from $B_0$ into $B_1$ for the range $H-L<n\le H$. Set $H_1=H$. }
\item[3.]{Iterate the following steps while $H_1<N$:
\begin{itemize}
\item[a.]{Set $h=H_1+\ell$.}
\item[b.]{In descending order $\Bigl\lceil\frac{h}{H}\Bigr\rceil>j\ge1$:
\begin{itemize}
\item[i.]{If $B_j$ has not been initialized, copy it from $B_0$ and set $H_j=H$.}
\item[ii.]{If $H_j<\ell\lceil h/j\ell\rceil$, increase it by $\ell$. Fill 
the bottom half of the new range with the values of the top half of the old range. 
Calculate new values as shown in Section
\ref{3:7}.}
\end{itemize}
}
\end{itemize}
}

\end{itemize}

Here we have the scheme of Fuller's program
 
\begin{python}[frame=none, xleftmargin = 0.5cm,
xrightmargin=0.5cm ]
<(*@\ref{3:2}@*) Fuller Program>  (*@$\equiv$@*)
<(*@\ref{3:9}@*) Some definitions>  
<(*@\ref{3:3}@*) setting parameters>
<(*@\ref{3:4}@*) initialize B[0]>
<(*@\ref{3:5}@*) creating the first running block>
<(*@\ref{3:8}@*) definition of the function Shift>
<(*@\ref{3:7}@*) definition of the function CalculateRunningBlock>
<(*@\ref{3:6}@*) Main Loop>
\end{python}

\subsection{Setting parameters.}\label{3:3}
As we have said we fix the three parameters $N$ or $\nMax$,  $H$
the length of the fixed block, and $\ell$ or  $\step$ being the size of the 
unit shift for the running blocks.  This is also half of the length $L$
of the running blocks. 
We use a function to check that the parameters satisfy the conditions 
\eqref{E:ineq}. The program prints the result of this check. Bad parameters
may cause an index error.

\begin{python}[frame=none, xleftmargin = 0.5cm,
xrightmargin=0.5cm ]
<(*@\ref{3:3}@*) Setting parameters>  (*@$\equiv$@*)
<(*@\ref{3:3:1}@*) def of checkparameters>
nMax = 1000000
H= 200000
step = 10000
L = 2*step
test = checkparameters(nMax,H,step)
if test:
    print 'GOOD PARAMETERS'
else:
    print 'NOT GOOD PARAMETERS'
\end{python}

\subsubsection{Checking parameters.}\label{3:3:1}
The  function `checkparameters$(\ )$' is defined. It checks whether our
parameters satisfy the inequalities 
\eqref{E:ineq}.  We apply Corollary \ref{kmaxbound} to check whether 
$\kMax\le \ell$. 

\begin{python}[frame=none, xleftmargin = 0.5cm,xrightmargin=0.5cm ]
<(*@\ref{3:3:1}@*) def of checkparameters>  (*@$\equiv$@*)
def checkparameters(nMax,H,step):
    test = True
    beta = 0.584962501
    if step < 2*nMax**beta:
        test = False
    if nMax > step*H:
        test = False
    if step*(H/step) != H:
        test = False
    if step*(nMax/step) != nMax:
        test = False
    if H < sqrt(L*nMax):
        test = False
    return test
\end{python}

\subsection{Initialize $B_0$.}\label{3:4}
The fixed block $B_0$ and the running blocks $B_j$ are members of a list
$B$. Here we simply compute the complexities contained in the fixed block
(using the function defined in Section \ref{Stwo}) and put 
them as the first term in the list. The successive terms of this list will 
be the running blocks. Hence we will have $B\,[\,0\,]\,[\,n\,]=\Vert n\Vert$
for $1\le n\le H$.

We also initialize $\Heads$. This will be a list with
$(H_j)_{j=0}^r$ as elements,  where the fixed block $B_0$ will contain the complexities
of $n$ for $1\le n\le H = H_0$, and $B_j$, when initialized,  
will be the list formed by the 
complexities of $n$ for $H_j-L <n\le H_j$. Hence we will have
$B\,[\,j\,]\,[\,n\,]=\Vert n+H_j-L+1\Vert$ for $0\le n\le L-1$.

\begin{python}[frame=none, xleftmargin = 0.5cm,
xrightmargin=0.5cm ]
<(*@\ref{3:4}@*) Initialize (*@$B_0$@*)>  (*@$\equiv$@*)
B = [complexity(H)]
Heads=[H]
\end{python}

\subsection{Creating the first running block.}\label{3:5}
All running blocks $B_j$ with $j\ge1$ are created with $H_j = H$. 
Later they will be
shifted by steps of length $\ell$ to finish at or just above $N/j$.

\begin{python}[frame=none, xleftmargin = 0.5cm,
xrightmargin=0.5cm ]
<(*@\ref{3:5}@*) Creating the first running block>  (*@$\equiv$@*)
<(*@\ref{3:5:1}@*) definition of CreateRunningBlock>
CreateRunningBlock(B,Heads)
\end{python}

\subsubsection{Definition of CreateRunningBlock.} \label{3:5:1}
The definition of this function is very simple. 
We simply copy the last half of the fixed block and put it as running 
block.  We  also have to update the Heads list.

\begin{python}[frame=none, xleftmargin = 0.5cm,
xrightmargin=0.5cm ]
<(*@\ref{3:5:1}@*) definition of CreateRunningBlock>  (*@$\equiv$@*)
def CreateRunningBlock(B,Heads):
    RB = B[0][H-L+1:H+1]
    B.append(RB)
    Heads.append(H)
\end{python} 

\subsection{Definition of the function Shift.}\label{3:8}
The running block $B_j$ contains the complexities of $n$ for 
$H_j-2\ell <n\le H_j$. While running the program this block is shifted  $\ell$ units
to the right. This is done in two steps, first the function Shift simply 
puts $H_j=H_j+\ell$, and copies the second half of the block $B_j$ into the first half
of this block and initializes the new values to any upper bound for $\Vert n\Vert$. 
Subsequently the function  'ComputeRunningBlock' will compute the 
new complexities $||n||$ for $H_j<n\le H_j+\ell$. 

\begin{python}[frame=none, xleftmargin = 0.5cm,
xrightmargin=0.5cm ]
<(*@\ref{3:8}@*) definition of the function Shift>  (*@$\equiv$@*)
def Shift(B,j):
    Heads[j] = Heads[j] + step
    B[j] = B[j] + [255 for n in range(0,step)]
    B[j] = B[j][step:3*step]
\end{python}

The most delicate point of the program will be to show that
the function ComputeRunningBlock can  complete its task. But we may assume
that in some way this has been done and analyze the Main Loop. This allows us 
to have a  picture of which running blocks have been initialized at a given 
point in the program and which values  the $H_j$ will have.

We have initialized the new values of $B_j$ to 255. We need that 
this value is an upper bound of $\Vert n\Vert$ for all $n\le \nMax$. Since
$\Vert n\Vert\le \frac{3}{\log2}\log n$ we may safely take $255$
for $\nMax\le 3.86\times 10^{25}$. The value 255 is useful because in this 
way we may take $B_j$ as a byte array.

\subsection{Main Loop.}\label{3:6}
The main computation is that of the first running block $B_1$, while $H_1$ 
will indicate  where we are in the task of computing all the complexities up to $\nMax$. 
Each run of the main loop consists essentially of a shift of this running 
block from $H_1$ to $H_1+\ell$.  Hence we have to compute 
$\Vert n\Vert$ for $H_1<n\le H_1+\ell$. For $n=ab$ we must read $\Vert a\Vert$
and $\Vert b\Vert$.  This is the purpose of the other running blocks and the 
fixed block. For small $a$, the running block $B_a$ must contain the values 
of $\Vert b\Vert$. The value of $\Vert a\Vert$ will be read from the fixed block.
For $n=a+b$ with $a<\kMax$ we read $\Vert a\Vert$ from the fixed block and 
$\Vert b\Vert$ from the running block $B_1$. 

It follows that we have to shift $B_a$, if needed, before computing $B_1$.
To update $B_a$ we follow the same procedure as to update $B_1$. 
Hence we proceed to update the $B_a$ in descending order of $a$. 
We  also have to initialize the needed running blocks.  

Here is the algorithm for the main loop. We will show that the necessary readings
can easily be performed.

\begin{python}[frame=none, xleftmargin = 0.5cm,
xrightmargin=0.5cm ]
<(*@\ref{3:6}@*) Main Loop>  (*@$\equiv$@*)
while Heads[1] < nMax:
    h = min(Heads[1]+step,nMax)
    for j in range(ceiling(h,H)-1,0, -1):
        newH = step*ceiling(h,(j*step))
        if notInitialized(j)and(newH > H):
            CreateRunningBlock(B,Heads)
        if Heads[j] < newH:
            Shift(B,j)
            CalculateRunningBlock(j,newH)
\end{python}

\begin{proposition}\label{P:6}
At the start of each run of the main loop, except the first,  
we will have initialized the 
running blocks $B_j$ for $1\le j< \lceil H_1/H\rceil$ and 
$H_j=\ell\lceil H_1/j\ell\rceil$.
\end{proposition}

\begin{proof}
At the start we have initialized only $B_1$ and $H_1=H$.
Then we put $h=H_1+\ell=H+\ell$. Since $1<\ell<H$\footnote{Since $\ell\mid H$, we 
have $\ell\le H$. If $\ell=H$, we have $\ell=H\ge\sqrt{LN}=\sqrt{2\ell N}$,
so that $\ell\ge2N$, contradicting $\ell\mid N$.}
we would have
$\lceil h/H\rceil=2$. The index $j$ in this first run will take only the value 
$j=1$. Since $B_1$ is initialized  we do not initialize any running block in this run. Since $H<\ell\lceil(H+\ell)/\ell\rceil=H+\ell$, the block
$B_1$ will be shifted and $H_1$ will be put equal to $H+\ell$.

Therefore, at the end of the first run of the loop we have $H_1=H+\ell$, so that 
$\lceil H_1/H\rceil=2$, and the only running block initialized is $B_1$, and 
$H_1=\ell\lceil H_1/\ell\rceil = H+\ell$. Therefore, the Proposition in true in this case.

Now by induction we may show that if our claim is true when a run starts, then it will 
be true at the end of this run.

At the start of a run we will have initialized the running block $B_j$ for
$1\le j<\lceil H_1/H\rceil$, and $H_j=\ell\lceil H_1/j\ell\rceil$. 

We will initialize in the next loop those $B_j$ not yet initialized, i.~e.~such 
that $\lceil H_1/H\rceil\le j$, and such that $\newH>H$, i.~e.~such that
$H< \ell\lceil(H_1+\ell)/j\ell\rceil$. Given  natural numbers $a$, $b$ and $c$, 
the relation  $a<\lceil b/c\rceil$ is 
equivalent to $ac<b$. Therefore, since $\ell\mid H$, the condition  
$H< \ell\lceil(H_1+\ell)/j\ell\rceil$ is equivalent to $Hj<H_1+\ell$ and this 
is equivalent to $j<\lceil(H_1+\ell)/H\rceil$.  Thus at the end of the loop 
all blocks $B_j$ with $1\le j<\lceil (H_1+\ell)/H\rceil$ will be initialized. 
Since the new $H_1$ will then be equal to $H_1+\ell$ we get half of our assertion. 

Now we have to show that at the end of the loop  $H_j$ will be equal to 
$H_j=\ell\lceil (H_1+\ell)/j\ell\rceil$.

For one of those blocks $B_j$ that were initialized at the start of the 
loop we have $H_j =\ell\lceil H_1/j\ell\rceil$. If we have
$H_j <\ell\lceil (H_1+\ell)/j\ell\rceil$ then this block is shifted in this
loop. If not, then we will have $H_j \ge\ell\lceil (H_1+\ell)/j\ell\rceil$.

In the first case $H_j/\ell=\lceil H_1/j\ell\rceil<\lceil (H_1+\ell)/j\ell\rceil$.
This is an instance of $\lceil x\rceil<\lceil x+y\rceil$ with $0<y\le 1$, so that
we will have $\lceil x+y\rceil=\lceil x\rceil+1$. Therefore, 
$\lceil (H_1+\ell)/j\ell\rceil=H_j/\ell +1$. It follows that 
$H_j+\ell = \ell\lceil (H_1+\ell)/j\ell\rceil$ and after the shift we will
have what we want. 

In the second case, the block has not been shifted, and 
$\ell\lceil (H_1+\ell)/j\ell\rceil\le H_j=\ell\lceil H_1/j\ell\rceil$. 
It follows that $H_j=\ell\lceil (H_1+\ell)/j\ell\rceil$. Again what we want. 

Now let $B_j$ be one of the running blocks that have just been initialized. So 
$j\ge \lceil H_1/H\rceil$ and $H <\ell\lceil (H_1+\ell)/j\ell\rceil$. After 
the initialization we have $H_j=H$, so that this block will be shifted. After
the shift $H_j$ is changed to $H+\ell$. Therefore, we must show that 
$H+\ell=\ell\lceil (H_1+\ell)/j\ell\rceil$.
First notice that $j\ge \lceil H_1/H\rceil$ is equivalent to 
$H\ge\ell\lceil H_1/j\ell\rceil$.  Therefore
\begin{displaymath}
\Bigl\lceil\frac{H_1}{j\ell}\Bigr\rceil\le \frac{H}{\ell} <\Bigl\lceil 
\frac{H_1}{j\ell}+\frac{1}{j}\Bigr\rceil.
\end{displaymath}
But since $0<1/j\le 1$ this implies 
\begin{displaymath}
\Bigl\lceil\frac{H_1}{j\ell}\Bigr\rceil=\frac{H}{\ell} < 
\Bigl\lceil\frac{H_1}{j\ell}+\frac{1}{j}\Bigr\rceil=\frac{H}{\ell}+1
\end{displaymath}
from which we get $H+\ell=\ell\lceil (H_1+\ell)/j\ell\rceil$.
\end{proof}

\subsection{Definition of the function CalculateRunningBlock.}\label{3:7}
After each application of the  function $\Shift(B,j)$, we must apply 
$\CalculateRunningBlock(j,\newH)$. This computes the improved values of the 
complexity for $H_j-\ell<n\le H_j$. Recall that the values of $B_j$ were initialized
to the upper bound $255$.

Calculating new values for $B_j$.
The steps are:
\begin{itemize}

\item[1.]{For each $2\le a\le b$, $H_j-\ell<ab\le H_j$:
\begin{itemize}
\item[a.]{Read $\Vert a\Vert$ from $B_0$. Read $\Vert b\Vert$ from $B_0$ if $b\le H$, otherwise from $B_{aj}$.}
\item[b.]{Set $\Vert a\Vert+\Vert b\Vert$ as the array value of $\Vert ab\Vert$ if it is lower than the current value.}
\end{itemize}
}
\item[2.]{For each $H_j-\ell<n\le H_j$ in ascending order:
\begin{itemize}
\item[a.]{Read $\Vert n-1\Vert+1$ from $B_j$ and set this as the array value of $\Vert n\Vert$ if it is lower than the current value.}
\item[b.]{Calculate $\kMax$ for $n$ using the current array value for $\Vert n\Vert$.}
\item[c.]{For each $6\le k\le kMax$:
\begin{itemize}
\item{Read $\Vert k\Vert$ from $B_0$ and $\Vert n-k\Vert$ from $B_j$. Set $\Vert n-k\Vert+\Vert k\Vert$ as the array value for $\Vert n\Vert$ if it is lower than the current value.}
\end{itemize}
}
\end{itemize}
}
\end{itemize}

\begin{python}[frame=none, xleftmargin = 0.5cm,
xrightmargin=0.5cm ]
<(*@\ref{3:7}@*) definition of the function CalculateRunningBlock>  (*@$\equiv$@*)
<(*@\ref{3:7:1}@*) definition of Products>
<(*@\ref{3:7:2}@*) definition of Sums>
def CalculateRunningBlock(j,newH):
    Products(j,newH)
    Sums(j,newH)
\end{python}
    
\subsubsection{Definition of Products.}\label{3:7:1} 
We consider all the products $ab$ where $2\le a\le b$, $H_j-\ell<ab\le H_j$.
We denote by $\Block$ the current running block, by $\BlockA$ and $\BlockB$ the 
blocks where $\Vert a\Vert$ and $\Vert b\Vert$, respectively, are to be found. 
At the same time we define $\shift$,  $\shiftA$ and $\shiftB$, in such a way that
$\Vert ab\Vert$ must be situated in $\Block\,[\,ab-\shift\,]$,
$\Vert a\Vert$ in $\BlockA\,[\,a-\shiftA\,]$ and 
$\Vert b\Vert$ in $\BlockA\,[\,b-\shiftB\,]$. 

\begin{python}[frame=none, xleftmargin = 0.5cm,
xrightmargin=0.5cm ]
<(*@\ref{3:7:1}@*) definition of Products>  (*@$\equiv$@*)
def Products(j,newH):
    Block = B[j]
    shift = newH-L+1
    BlockA = B[0]
    shiftA = 0
    a = 2
    b = max(1+(newH-step)/a, a)
    while (a*a <= newH) and (b >= a):
        if newH/a <= H:
            BlockB = BlockA
            shiftB = shiftA
        else:
            BlockB = B[a*j]
            shiftB = Heads[a*j]-L+1
        bmax = newH/a
        ab = a*b
        while b <= bmax:
            if Block[ab-shift] > BlockA[a-shiftA] + BlockB[b-shiftB]:
                Block[ab-shift] = BlockA[a-shiftA] + BlockB[b-shiftB]
            b = b+1
            ab = ab+a
        a = a+1
        b = max(1+(newH-step)/a, a)
\end{python} 

\begin{proposition}
The procedure $\Products(j,\newH) $ is correct, i.~e.~the needed complexities
$\Vert a\Vert $ and $\Vert b\Vert$ are contained in the indicated blocks.
\end{proposition}

\begin{proof}
1. Proof that $a\le H$ so that $\Vert a\Vert$ can always be read from the fixed block.

Since $a^2\le ab\le \newH$ we have $a\le \sqrt{\newH}$.  On the other hand
we have $\newH\le N+\ell$, because we only compute the complexities up to $N$. 
Then by \eqref{E:ineq} 
\begin{displaymath}
a\le \sqrt{\newH}\le  \sqrt{N+\ell}<\sqrt{N L}\le H.
\end{displaymath}
\medskip

2. Proof that $\frac{newH}{a}\le H$ or that $B_{aj}$ has been initialized. 

Assume that $B_{aj}$ is not initialized. 
Since we are shifting the block $B_j$ we are in  step $j$ of the 
Main Loop \ref{3:7}. The first not initialized running block will be 
$B_u$ with $u=\lceil h/H\rceil$, so that $j<u\le aj$, since $B_j$ is
initialized and $B_{aj}$ is  not. From $\lceil h/H\rceil\le a j$ we get
$h/H\le aj$, so that $h/j\le aH$.

Since $B_j$ has already been shifted if needed,
we have by Proposition \ref{P:6},  that $H_j=\newH=\ell\lceil h/j\ell\rceil$. 
Because $\ell\mid H$ it follows that
\begin{displaymath}
\newH=\ell\Bigl\lceil \frac{h}{j\ell}\Bigr\rceil\le 
\ell\Bigl\lceil aH\frac{1}{\ell}\Bigr\rceil = aH.
\end{displaymath}

Therefore, $b\le \frac{\newH}{a}\le H$. Hence if $B_{aj}$ is not initialized, 
then  $\frac{\newH}{a}\le H$ and we may
read $\Vert b\Vert$ from the fixed block $B_0$. But the program chooses $\BlockB=B_0$
when $ \frac{\newH}{a}\le H$.

\medskip

3. If $B_{aj}$ is initialized. By Proposition \ref{P:6}, we have in this 
case $aj<\lceil h/H\rceil$ and $\newH=\ell\lceil h/j\ell\rceil$. This implies
respectively that $aj<h/H$, and (since $\ell\mid\newH$) $h/j\ell\le \newH/\ell$.
It follows that
\begin{displaymath}
aH<\frac{h}{j}\le \newH\quad\text{so that}\quad H<\frac{\newH}{a}.
\end{displaymath}
Hence in this case the program puts $\BlockB =B_{aj}$. 
In fact we can read $\Vert b\Vert$ from $B_{aj}$ because
\begin{displaymath}
H_{aj}-\ell\le \frac{\newH-\ell}{a}<b\le \frac{\newH}{a}\le H_{aj}.
\end{displaymath}
The two intermediate inequalities  are true because we have by hypothesis
$\newH-\ell<ab\le \newH$. 

To prove the first inequality 
put $h=j\ell p-r$ with $0\le r<j\ell$ and $p=aq-s$ with $0\le s<a$. Then 
$h= aj\ell q -(sj\ell+r)$ and $0\le sj\ell+r <(s+1)j\ell\le aj\ell$, 
so that
\begin{multline*}
H_{aj}-\ell=\ell\Bigl(\Bigl\lceil\frac{h}{aj\ell}\Bigr\rceil-1\Bigr)=
\ell (q-1)=\ell\frac{aq-s-(a-s)}{a}\le\\ \le\frac{\ell}{a}(p-1)= \frac{\ell}{a}\Bigl(\Bigl\lceil\frac{h}{j\ell}\Bigr\rceil-1\Bigr)
=\frac{\newH-\ell}{a}.
\end{multline*}
With the same notation
\begin{displaymath}
\frac{\newH}{a}=\frac{\ell}{a}\Bigl\lceil\frac{h}{j\ell}\Bigr\rceil
=\frac{\ell}{a}p=\frac{\ell}{a}(aq-s)\le\ell q=\ell\Bigl\lceil\frac{h}{aj\ell}
\Bigr\rceil=H_{aj}.
\end{displaymath}

Hence if $B_{aj}$ is initialized then $\Vert b\Vert$ can be read from $\BlockB$ which
in this case will be $B_{aj}$, since in this case,  $H<\frac{\newH}{a}$.
\end{proof}     
  
\subsubsection{Definition of Sums.}\label{3:7:2}      
This is simpler than the case of the products.  The definition of 
the function $\kMaxfor(n,s)$ should be clear after our explanation in Section
\ref{2:7}. The definition of the function $E(\cdot)$ is contained in Section
\ref{2:3}.

As explained in Section \ref{2:8} we only have to test the values of $k\le \kMax$ which
are solid numbers. We precompute the first few of  these numbers and put them in 
a file in the form $\text{solid} = [\,1, 6, 8, 9, 12, \dots  \,]$.
If this file is not large enough an index-error  will appear. In practice this will
not be a 
problem  because $\kMax$ is relatively small and the sequence of solid
numbers is easy to compute.

\begin{python}[frame=none, xleftmargin = 0.5cm,
xrightmargin=0.5cm ]
<(*@\ref{3:7:2}@*) definition of Sums>  (*@$\equiv$@*)
from solid import *
def kMaxfor(n,s):
    target = s
    t = target/2
    while E(t)+E(target-t) < n:
        t = t-1
    return E(t)
def Sums(j,newH):
    for n in range(newH-step+1,newH+1):
        Block = B[j]
        shift = newH-L+1
        s = Block[n-1-shift]+1
        m = n
        kMax = kMaxfor(m,s)
        Block0 = B[0]
        r = 0
        b = solid[r]
        while b <= kMax:
            if Block[n-shift] > Block[n-b-shift]+Block0[b]:
                Block[n-shift] = Block[n-b-shift]+Block0[b]
            r = r+1
            b = solid[r]
\end{python}

Since we have taken $\ell>\kMax$ for all values of $n$ in $1\le n\le N$ 
it is clear that $\Vert n-k\Vert$ can  always be read from the current running 
block. Indeed, since the complexities $\Vert n\Vert$  for $\newH-\ell<n\le \newH$ are
computed in increasing order and the values in the first 
half of the running block (which 
has length $\ell$) are correct from the start.  Also the value 
of $\Vert k\Vert$ can always be read from the fixed block  since 
$k\le \kMax\le \ell\le H$.

\subsection{Some definitions}\label{3:9}
There are two simple functions in the main loop that we have not yet defined. 
The ceiling function is used  for certain ranges in the Main Loop.   To check 
whether a running block is initialized we use 
the length of the Heads list. Each time a running block is initialized
we put a new element in this list. 

\begin{python}[frame=none, xleftmargin = 0.5cm,
xrightmargin=0.5cm ]
<(*@\ref{3:7:2}@*) Some definitions>  (*@$\equiv$@*)
def ceiling(n,m):
	return n/m+(n
def notInitialized(j):
    return len(Heads) < j+1
\end{python}

\section{Upper bounds for $\Vert n\Vert$.}\label{S:UB}
To estimate the running time of our algorithms we need upper bounds for 
$\Vert n\Vert$. We get some useful upper bounds by expressing $n$ in a 
base $b$ and using Horner's algorithm.  For example
$n= r_0r_1r_2\dots r_k$  where $0\le r_j<b$ and $r_0\ge1$.  This is equivalent to 
\begin{equation}\label{basicform}
n=r_k+r_{k-1}b+r_{k-2}b^2+\cdots r_0 b^k=r_k+b(r_{k-1}+b(r_{k-2}+\cdots+b(r_1+ b r_0)\cdots)).
\end{equation}

For each digit $0\le r<b$ we write $D(b,r)$ for the complexity of 
multiplying by $b$ and adding $r$. E. g.  $D(6,4)\le 7$ because
$\Vert 6n+4\Vert=\Vert 3(2n+1)+1\Vert\le \Vert n\Vert+3+2+1+1=\Vert n\Vert+7$. 
In general we  define $D(b,r)$ as the least number  satisfying
\begin{equation}
\Vert r+bn\Vert\le \Vert n\Vert+D(b,r)
\end{equation}
for all $n\ge1$. 

We will always have $D(b,r)\le \Vert b\Vert+\Vert r\Vert$. 
This inequality may be strict. For example we have just seen that 
$D(6,4)\le 7$ and $\Vert 6\Vert+\Vert 4\Vert=9$.

\begin{proposition}
If the expansion of $n$ in base $b$ is given by $n= r_0r_1r_2\dots r_k$,
then 
\begin{equation}\label{boundbase}
\Vert n\Vert\le \Vert r_0\Vert+\sum_{j=1}^k D(b, r_j).
\end{equation}
\end{proposition}
\begin{proof}
For numbers with one digit this is trivially true. In the other case we have
$n = r_k+mb$ with $m=r_0r_1r_2\dots r_{k-1}$ so that by  induction  we get
\begin{displaymath}
\Vert n\Vert\le \Vert m\Vert+D(b,r_k)=\Vert r_0\Vert+\sum_{j=1}^{k-1} D(b,r_j)+ D(b,r_k).
\end{displaymath}
\end{proof}
 
By means of the following proposition we may easily obtain upper bounds for 
$D(b,r)$. $\phantom{0}$

\begin{proposition}
Let $d\mid b$  with $1<d<b$, where $b=da$.   For $1\le r\le b$ put 
$r=qd+s$. Then 
\begin{equation}
D(b,r)=D(d a, qd+s)\le D(d,s)+D(a,q).
\end{equation}
\end{proposition}

\begin{proof}
We have
\begin{displaymath}
\Vert n b+r\Vert=\Vert (n a +q)d+s\Vert=D(d,s)+\Vert na+q\Vert\le 
D(d,s)+D(a,q)+\Vert n\Vert.
\end{displaymath}
\end{proof}

We  define a function 
$D_0(b,r)$ for $b\ge2$ and $0\le r<b$ by induction on $b\ge2$.

First when $b=p$ is prime
\begin{equation}
D_0(p,r):=\Vert p\Vert+\Vert r\Vert
\end{equation}
where we take $\Vert 0\Vert :=0$, and in general
\begin{equation}
D_0(b,r):=\min\bigl\{\Vert b\Vert+\Vert r\Vert, \min_{1<d<b, d\mid b}D_0(d,\bmod(r,d))
+D_0(b/d,\lfloor r/d\rfloor)\bigr\}.
\end{equation}

By induction we will then find that
\begin{equation}
D(b,r)\le D_0(b,r).
\end{equation}

\begin{remark}
The best bounds are usually obtained with bases of the form $b=2^n 3^m$. Notice that 
to compute $D(b,r)$ we need to precompute  only the values of $D(b',r)$ for 
all proper divisors $b'\mid b$. 
\end{remark}

\subsection{Average bound of the complexities.}
We will give here an application to the general theory of the complexity 
of natural numbers. 
We may define several constants associated with  the bounds on the complexity. 
\begin{definition}
Let $C_{\text{max}}$ denote the $\limsup_{n\to\infty}\frac{\Vert n\Vert}{\log n}$. Let 
$C_{\text{avg}}$ be
the infimum of all $C$ such that 
\[ \Vert n\Vert \le C\log n\]
for a set of natural numbers  of density $1$.
\end{definition}
We have 
\begin{displaymath}
\frac{3}{\log 3}\le C_{\text{avg}}\le C_{\text{max}}\le \frac{3}{\log 2}.
\end{displaymath}

In \cite{G} it is said that Isbell has shown using the expression of $n$ in basis 
24 that $C_{\text{avg}}\le 3.475$.  In \cite{S} Steinerberger considered also a related problem 
obtaining for a slightly different constant  the bound
$C'_{\text{avg}}\le 3.332$.

\begin{proposition}
For any basis $b\ge 2$ we have
\begin{equation}\label{averagebound}
C_{\text{avg}}\le \frac{1}{b\log b}\sum_{r=0}^{b-1}  D(b,r).
\end{equation}
In particular we get 
\[C_{\text{avg}}=\frac{41\,747\,875}{2^{7}3^8\log(2^{9}3^8)}=
3.3080772123153688960\dots\dots  \le  3.309.\] 
\end{proposition}

\begin{proof}

Given a number in base $b$ as $n=r_0r_1\dots r_k$ with $r_0\ne 0$, then 
$n\ge b^k$ so that by \eqref{boundbase} we have
\[\frac{\Vert n\Vert}{\log n}\le \frac{\Vert r_0\Vert}{k\log b}+\frac{1}{k\log b}
\sum_{j=1}^k D(b,r_j).\]
The first term tends to $0$ as $n$ tends to $\infty$. By Chernoff's Theorem,   for almost all numbers 
the difference
\[\frac{1}{k}
\sum_{j=1}^k D(b,r_j)-\frac{1}{b}\sum_{r=0}^{b-1}D(b,r)\]
is small. 
This proves \eqref{averagebound}.

We may easily compute the numbers $\frac{1}{b}\sum_{r=0}^{b-1}  D(b,r)$, for increasing 
bases. The smaller values are obtained for basis of the form $b=2^n3^m$.  Having 
computed all 
numbers $D(b,r)$ for all  bases $2^n 3^m <3359232=2^9 3^8$, we get the best value 
for $b=2^9 3^8$ for which
\[\frac{1}{b\log b}\sum_{r=0}^{b-1}  D(b,r)= 3.3080772123153688960\dots.\]
\end{proof}

\section{Performance of the time-improved algorithm.}\label{S:PF}

\begin{theorem}
The time-improved algorithm presented in Section \ref{S:2}  computes all $\Vert n\Vert$ for $ n\le N$ in time $\Orden(N^{1.230175})$ 
and space $\Orden(N\log\log N)$.
\end{theorem}

\begin{proof}
Clearly we need space for the array $\Compl$. At the start of the algorithm 
we initialize this array by a common value $\cMax$ of the order 
$\log N$ (observe that to simplify the notation we have put $N=\nMax$), and we need 
$\log\log N$ bits to store this value. During the calculation each element 
of the array is decreased, so that  we need $\Orden(N\log\log N)$ bits of space. 
The space needed for the rest of the computation is only $\Orden(\log N)$, needed to 
store some numbers  $\le N$ or to compute $E(t)$ for $t$ of the order of $\log N$. 

The test for the products takes a number of operations of the order
\begin{displaymath}
\sum_{n=1}^N \min(n,N/n)=\sum_{n=1}^{\sqrt{N}}n+\sum_{n=\sqrt{N}}^N \frac{N}{n}=
\Orden(N)+\Orden(N\log N)=\Orden(N\log N).
\end{displaymath}

The cost of computing $\kMax$ for a value of $n$ is $\Orden(\log^2n)$ operations.
In fact the starting value of $k$ is of the order $\Vert n-1\Vert$, so of order
$\log n$. In the procedure $k$ will be changed at most $k$ times and for each of these
we have to compute two values of $E$ for numbers of size $\log n$ 
at a cost of $\Orden(\log n)$ operations and $\Orden(\log n)$ space.

In the main loop for each $n\le \text{nMax}$ we have precomputed $\Vert n-1\Vert$. We  compute $\kMax$ for such an $n$ and  then in the check of sums we run the variable $m$  
from $6$ to $\kMax$ requiring a fixed number of operations for each value of $m$.  The cost of all this is of the order of
\begin{displaymath}
C=\sum_{n=1}^N (\log^2n +\kMax(n))=\Orden(N\log^2N)+\sum_{n=1}^N n
\Bigl(1-\sqrt{1-\frac{4}{n^2}3^{\Vert n-1\Vert/3}}\Bigr).
\end{displaymath}
Since
\begin{displaymath}
4\frac{3^{\Vert n-1\Vert/3}}{n^2}\le 4
\frac{\exp\left(\frac{\log3}{3}\frac{3}{\log2}\log n
\right)}{n^2}\le  4 \,n^{-0.415037}
\end{displaymath}
tends to $0$ when $n\to\infty$ we may bound the cost by
\begin{displaymath}
C\le \Orden(N\log^2N)+\Orden\Bigl(\sum_{n=1}^N \frac{3^{\Vert n-1\Vert/3}}
{n}\Bigr).
\end{displaymath}
Now choose a base $b$ and apply the bound \eqref{boundbase}.  Taking a number
$a$ such that $b^{a-1}\le N< b^a$  we will have
\begin{displaymath}
\Orden\Bigl(\sum_{n=1}^N \frac{3^{\Vert n-1\Vert/3}}
{n}\Bigr)\le 
\Orden\Bigl(\sum_{\ell=1}^{a}\sum_{b^{\ell-1}\le n< b^\ell}
\frac{3^{\Vert n\Vert/3}}
{n}\Bigr)=\Orden_b\Bigl(\sum_{\ell=1}^{a}b^{-\ell}\sum_{b^{\ell-1}\le n< b^\ell}
3^{\Vert n\Vert/3}\Bigr).
\end{displaymath}
Now in the inner sum $n$ runs through all the numbers that in base $b$ have $\ell$ digits. 
By \eqref{boundbase} we will have
\begin{displaymath}
\sum_{b^{\ell-1}\le n< b^\ell}
3^{\Vert n\Vert/3}\le C_b\sum_{b^{\ell-1}\le n< b^\ell}
 3^{\frac13\sum_{j=1}^\ell D(b,b_j)}\le C_b 
 \Bigl(\sum_{d=0}^{b-1}3^{\frac13D(b,d)}\Bigr)^\ell
=C_b A_b^\ell.
\end{displaymath}

Hence
\begin{displaymath}
C\le \Orden(N\log^2N)+
\Orden_b\Bigl(\sum_{\ell=1}^{a}(A_b/b)^\ell\Bigr)=\Orden(N\log ^2N)+
\Orden_b((A_b/b)^a).
\end{displaymath}
Since $N\sim b^a$ with a constant only depending on $b$
\begin{displaymath}
(A_b/b)^a=\exp\bigl(a\log(A_b/b)\bigr)=b^{a\frac{\log(A_b/b)}{\log b}}=\Orden(N^\alpha),\end{displaymath}
where
\begin{displaymath}
\alpha=\frac{\log(A_b/b)}{\log b}=-1+\frac{1}{\log b}\log\Bigl(\sum_{d=0}^{b-1}3^{\frac13D(b,d)}\Bigr).
\end{displaymath}
We have computed $\alpha$  for all bases $b=2^n 3^m\le 3\,188\,246$.
For $b=2\,239\,488=2^{10}3^7$ we found the smallest value 
\begin{multline*}
\alpha=\frac{\log(3^6 2^{-10}(30\,357\,189 + 21\,079\,056\cdot 3^{1/3} + 14\,571\,397 \cdot 3^{2/3})}{\log (2^{10}3^7)}=\\1.230\,174\,997\,215\,298\,061\,586\,\dots
< 1.230175.\end{multline*}
\end{proof}

\section{Performance of Fuller's algorithm.}

\begin{proposition}
The space-improved algorithm presented  in Section \ref{S:F} 
computes $\Vert n\Vert$ for all $ n\le N$ in time $\Orden(N^\alpha)$ using 
$\Orden(N^{(1+\beta)/2}\log\log N)$ bits of storage. 
( $\alpha=1.230175$ and $(1+\beta)/2)\approx0.792481$. )
\end{proposition}

\begin{proof}
In the last run of the Main Loop we have $H_1\le N+\ell$ so that the total number of 
initialized running blocks will be $\le \lceil(N+\ell)/H\rceil\le  N/H+2$. 

Hence we need a fixed block of length $H$ and $\approx N/H$ running blocks of 
length $L$.  
Each entry in the blocks must contain a value of the complexity, 
each of $\log\log N$ bytes. So the required space will be
\begin{displaymath}
\le \Bigl(H+\frac{N}{H}L\Bigr)\log\log N.
\end{displaymath}

The value of $L$ is limited by our conditions $\kMax\le \ell$ and $L=2\ell$. Therefore, 
an $L$ of the order  $\Orden(N^\beta)$ or larger would be adequate by Corollary 
\ref{kmaxbound}. 
Given $N$ and $L$ the best choice of $H$ (the one requiring less space) 
is  $H=\sqrt{LN}$, and this will give a  space requirement of $\sqrt{NL}\log\log N$.  
So the best choice will be to take 
$L=\Orden(N^\beta)$,  and $H=\Orden(N^{(1+\beta)/2})$.  It is easy to see that 
increasing $H$ and $N$ a little (if needed) we may also satisfy the conditions 
$\ell\mid H$ and $\ell\mid N$.

So the space requirement for the algorithm is $\Orden(N^{(1+\beta)/2}\log\log N)$, 
and this choice will satisfy all the conditions in \eqref{E:ineq}.

The time and the number of operations  needed for the computation is 
as follows

1. Computing the fixed block takes $\Orden(H^\alpha)$ operations. 

2. Copying from $B_0$ into $B_1$ takes $\Orden(L)$ operations.

The running block $B_j$ starts at $H$ and ends at $\approx N/j$ in steps of size 
$\ell$. Hence $B_j$ must be adapted about $N/j\ell$ times.  

There are $N/H$ running blocks  each of length $L$. Initializing all these requires \begin{displaymath}
\le C \sum_{j\le N/H} \frac{N}{j}=\Orden(N\log N/H)\qquad \text{operations}.
\end{displaymath} 

To compute the new values for the block $B_j$ (for the products) requires a fixed 
number of operations for each $2\le a\le b$ with $H-\ell<ab\le N/j$. 
The total cost of  the products is therefore
\begin{displaymath}
\le C \sum_{j\le N/H} \sum_{a\le \sqrt{N/j}} \frac{N}{aj}\le C \sum_{j\le N/H}\frac{N}{j}
\log\sqrt{N/j}\le CN\log N\log (N/H) \le \Orden(N\log^2N).
\end{displaymath}

To compute the new values for the block $B_j$ (for the sums) we have to perform 
the same number of operations as in the time-improved algorithm of Section \ref{S:2}. This cost was $\Orden(N^\alpha)$. Hence the cost for the sums in the block $B_j$ is 
at most   $C (N/j)^\alpha$.  It follows that all the sums for the block cost 
\begin{displaymath}
\le C \sum_{j\le N/H} \Bigl(\frac{N}{j}\Bigr)^\alpha\le C\zeta(\alpha) N^\alpha=
\Orden(N^\alpha).
\end{displaymath}
So the total cost of the algorithm is 
\begin{displaymath}
\Orden(N^\alpha)+\Orden(N\log^2N)+\Orden(N\log(N/H))=\Orden(N^\alpha).
\end{displaymath}
\end{proof}

\section{Acknowledgement}
We thank  H. Altman for pointing out the possible application of our bounds
in terms of $D(b,r)$ to improve upon existing bounds of $C_\text{avg}$.

\end{document}

%% file: pythonlisting.tex
\usepackage{color}
\usepackage[procnames]{listings}
\usepackage{textcomp}
\usepackage{setspace}

\definecolor{gray}{gray}{0.5}
\definecolor{green}{rgb}{0,0.5,0}
\definecolor{lightgreen}{rgb}{0,0.7,0}
\definecolor{purple}{rgb}{0.5,0,0.5}
\definecolor{orange}{rgb}{1,0.65,0}
\definecolor{darkred}{rgb}{0.5,0,0}
\lstnewenvironment{python}[1][]{
\lstset{
language=python,
basicstyle=\ttfamily\tiny\setstretch{1},
stringstyle=\color{green},
showstringspaces=false,
alsoletter={1234567890},
otherkeywords={\ , \}, \{},
keywordstyle=\color{blue},
emph={access,and,as,break,class,continue,def,del,elif,else,%
except,exec,finally,for,from,global,if,import,in,is,%
lambda,not,or,pass,print,raise,return,try,while,assert},
emphstyle=\color{orange}\bfseries,
emph={[2]self},
emphstyle=[2]\color{gray},
emph={[4]ArithmeticError,AssertionError,AttributeError,BaseException,%
DeprecationWarning,EOFError,Ellipsis,EnvironmentError,Exception,%
False,FloatingPointError,FutureWarning,GeneratorExit,IOError,%
ImportError,ImportWarning,IndentationError,IndexError,KeyError,%
KeyboardInterrupt,LookupError,MemoryError,NameError,None,%
NotImplemented,NotImplementedError,OSError,OverflowError,%
PendingDeprecationWarning,ReferenceError,RuntimeError,RuntimeWarning,%
StandardError,StopIteration,SyntaxError,SyntaxWarning,SystemError,%
SystemExit,TabError,True,TypeError,UnboundLocalError,UnicodeDecodeError,%
UnicodeEncodeError,UnicodeError,UnicodeTranslateError,UnicodeWarning,%
UserWarning,ValueError,Warning,ZeroDivisionError,abs,all,any,apply,%
basestring,bool,buffer,callable,chr,classmethod,cmp,coerce,compile,%
complex,copyright,credits,delattr,dict,dir,divmod,enumerate,eval,%
execfile,exit,file,filter,float,frozenset,getattr,globals,hasattr,%
hash,help,hex,id,input,int,intern,isinstance,issubclass,iter,len,%
license,list,locals,long,map,max,min,object,oct,open,ord,pow,property,%
quit,range,raw_input,reduce,reload,repr,reversed,round,set,setattr,%
slice,sorted,staticmethod,str,sum,super,tuple,type,unichr,unicode,%
vars,xrange,zip},
emphstyle=[4]\color{purple}\bfseries,
upquote=true,
morecomment=[s][\color{lightgreen}]{"""}{"""},
commentstyle=\color{red}\slshape,
literate={>>>}{\textbf{\textcolor{darkred}{>{>}>}}}3%
         {...}{{\textcolor{gray}{...}}}3,
procnamekeys={def,class},
procnamestyle=\color{blue}\textbf,
framexleftmargin=1mm, framextopmargin=1mm, frame=shadowbox,
rulesepcolor=\color{blue},#1
}}{}

%% file: 97-Complexity-v4.bbl
\begin{thebibliography}{99}
   


\bibitem{AZ1} 
\textsc{H. Altman, J. Zelinsky}, \emph{
Numbers with integer complexity close to the lower bound},
Integers \textbf{12} (2012) 1093-1125.


   
\bibitem{A}
\textsc{J. Arias de Reyna}, 
\emph{Complejidad de los números naturales}, Gaceta de la Real Sociedad 
Matemática Española \textbf{3} (2000) 230--250. 

   
\bibitem{AL}
\textsc{J. Arias de Reyna, J. van de Lune},  \emph{The question ``How many 
$1$'s are needed?'' revisited}, (2009) arXiv:1404.1850. 

\url{http://arxiv.org/abs/1404.1850}
   
\bibitem{G}
\textsc{R.~K.~Guy}, 
\emph{What is the least number of ones needed to represent $n$ using only
$+$ and $\times$ (and parentheses)?}, American Mathematical Monthly \textbf{93}
(1986) 189--190.
   
\bibitem{G2}
\textsc{R.~K.~Guy}, 
\emph{Unsolved Problems in Number Theory}, Third edition, Springer-Verlag,
New York, 2004.

\bibitem{JI}
\textsc{J. Iraids, K. Balodis, J. 	\v Cer\c nenoks, M. Opmanis, R. Opmanis, K.
Podnieks}, \emph{Integer Complexity: Experimental and Analytical results}, 
Scientific papers University of Latvia, Computer Science and Information Technologies,
\textbf{787} (2012) 153--179. 
arXiv:1203.6462, (2012). 

\url{http://front.math.ucdavis.edu/1203.6462}

\bibitem{MNF}
\textsc{M. N. Fuller}, \emph{C-program to compute A005245}, February 2008.

\url{http://oeis.org/A005245/a005245.c.txt}

\bibitem{K}
\textsc{D. E. Knuth}, \emph{Literate Programming}, The Computer Journal
\textbf{27} (1984) 97--111. 

\bibitem{MP}
\textsc{K.~Mahler,  J.~Popken},
\emph{On a maximum problem in arithmetic}, (in Dutch),
Nieuw Arch. Wiskunde (3) \textbf{1} (1953) 1--15.
   
   
\bibitem{R}
\textsc{D.~A.~Rawsthorne}, 
\emph{How many $1$'s are needed?},
Fibonacci Quart. \textbf{27} (1989) 14--17.


\bibitem{SS}
\textsc{V. V. Srinivas, B. R. Shankar}, \emph{Integer Complexity: Breaking the 
$\Theta(n^2)$ barrier}, World Academy of Science \textbf{48} (2008) 690--691.

\bibitem{S}
\textsc{S. Steinerberger }, \emph{A short note on Integer Complexity}, 
Contribution to Discrete Mathematics, to appear.

\bibitem{OEIS}
\textsc{OEIS Foundation Inc}, \emph{The On-Line Encyclopedia of Integer Sequences}, (2012).

\url{http://oeis.org}.
   
   
\end{thebibliography}
